\documentclass[12pt]{amsart}
\usepackage{amsmath}
\usepackage{amsthm}
\usepackage{amsfonts}
\usepackage{tikz}
\usepgflibrary{arrows}
\def\ods#1{\noindent\hbox to 1truecm{\hfil#1\ }\hangindent1truecm\hangafter1}
\def\re{\mathbb R}
\def\rn{\mathbb N}
\def\lm{Lebesgue measure}

\def\rmr{\re\times(\re\setminus\{0\})}

\def\card{\operatorname{card\,}}

\newtheorem{proposition}{Proposition}
\newtheorem{theorem}[proposition]{Theorem}
\newtheorem{lemma}[proposition]{Lemma}

\theoremstyle{remark}
\newtheorem{example}[proposition]{Example}
\newtheorem*{remark}{Remark}
\begin{document}

\author{Miroslav Chleb\'\i k}

\address{University of Sussex, UK} 
\email{m.chlebik@sussex.ac.uk}

\title{On the Erd\"os similarity problem}

\begin{abstract} New partial results are obtained related to the following old problem of Erd\"os:  for any
infinite set $X\subseteq \re$ to show that there is always a measurable (or,
equivalently, closed) subset of $\re$ of positive Lebesgue measure which contains no
subset geometrically similar to~$X$.\end{abstract}



\maketitle

\section{Introduction}

A subset $X$ of real numbers is called {\it universal (modulo similarities) w.r.t.\ the sets of positive measure} if every Lebesgue measurable set of positive measure necessarily contains a geometrically similar copy of ~$X$; in $\re$ that means a copy of $X$ by a nonconstant affine map. It is an easy observation based on Lebesgue density theorem that all finite sets are universal. A long standing conjecture of Erd\"os is that there is no infinite universal set. In another words, that every infinite set $X \subseteq \re$ is non-universal in a sense that for such a set there exists a Lebesgue measurable set $C \subseteq \re$  of positive measure not containing any similar copy of~$X$.

Paul Erd\"os first posed this problem in 1974 at The Fifth Balkan Mathematical Congress, and then he repeated it on many occasions. The problem is recorded in \cite{E} in 1981, where he offers a prize of $\$100$ for the solution. While this naturally arising question was asked more than 40 years ago, there are very few results concerning it and the general result remains unproved. It is not even known if it is true for every uncountable set~$X$. Many useful references can be found in \cite{S} and \cite{CFG}.

Given any set~$X$, if some subset of~$X$ is non-universal, then so is~$X$. Moreover, the classes of universal/non-universal sets are invariant under the similarity mappings. Therefore much attention has been focused on the case where~$X$ is a {\it zero-sequence}; namely, an infinite strictly monotone sequence of positive numbers converging to $0$. If every zero-sequence is proven to be non-universal, then the problem is solved. On the other hand it is not even known whether all uncountable sets are non-universal, and so far we have not been able to exclude that Cantor-like universal sets exist.

Various authors have proved the conjecture for zero-sequences that converge sufficiently
slowly, for example Falconer~\cite{F}, with $\lim_{k\to
\infty}\frac{x_{k+1}}{x_k}=1$. Bourgain \cite{B} tackled the problem from a different perspective,
disproving first an infinite version of the $3$-dimensional Szemer\'edi phenomenon. As a corollary he shows non-universality of any triple sum
$X=X_1+X_2+X_3$, where $X_1$, $X_2$, and $X_3$ are infinite sets of reals.


The above notion of universality w.r.t.\ the collection of all sets of positive Lebesgue measure is trivial for unbounded sets; any unbounded subset of $\re$ is non-universal then (witnessed by any bounded measurable set of positive measure). It is interesting to mention that for any bounded set ~$X$ the above notion of universality w.r.t.\ the sets of positive Lebesgue measure is equivalent to the one w.r.t.\ the sets whose complement has finite Lebesgue measure. Consider a bounded non-universal set $X \subseteq \re$, and let~$C$ be a witness of its non-universality; namely,~$C$ is a Lebesgue measurable set of positive measure containing no similar copy of~$X$. Using a sequence ~$C_n$ of similar copies of $C$, properly scaled about a point of density of~$C$ so that the Lebesgue measure of its relative complement in $(-n, n)$, namely of $G_n:= (-n, n) \setminus C_n$,  is very small, say $\lambda (G_n)<\epsilon2^{-n}$, one can construct another such witness whose complement $G=\cup_{n=1}^\infty G_n$ is of arbitrarily small Lebesgue measure.

That is why we address the question of  {\it universality (modulo similarities) w.r.t.\ the sets with complement of finite Lebesgue measure}; for bounded sets $X \subseteq \re$ we are back to the original Erd\"os problem, but this new setting does the problem to characterize universal sets interesting for unbounded sets as well. This kind of questions could be viewed as an attempt to address Szemer\'edi-type problems (about finite patterns in sets of integers of positive upper density) also in the continuous setting, and for infinite structures.

We need to develop the methods of proofs that certain infinite patterns $X \subseteq \re$
are non-universal. For such an $X$ we need to find (or to prove the existence of) a witness of non-universality of~$X$; namely, a Lebesgue measurable set $C \subseteq \re$ with complement of finite Lebesgue measure that contains no similar copy of~$X$. As any such measurable set can be approximated in measure from inside by its closed subsets, we will see that then also a closed witness exists. But for a closed set~$C$, the question of whether it contains any similar image of~$X$ or not is equivalent to the same being true for similar images of the closure of~$X$. That is why a set $X \subseteq \re$ is universal if and only if its closure is universal. Moreover, by the scale invariance of the problem, the complement of the witness can be an open set of arbitrarily small positive Lebesgue measure.

{\bf Non-universal sets.} A set $X \subseteq \re$ is non-universal w.r.t.\ the sets with complement of finite Lebesgue measure if and only if for every $\varepsilon>0$ there is a measurable (or, equivalently, open) set $G$ with $\lambda(G)<\varepsilon$, that intersects all similar copies $(a+bX)$ of~$X$.

{\bf Reformulation as a plane covering problem.} One can represent a similarity mapping $X\mapsto a+bX$ by a point $(a, b)$ in $\re\times(\re\setminus\{0\})$.  For any subset of similarities $A \subseteq \re\times(\re\setminus\{0\})$ and for any set $G \subseteq \re$,
$(a+bX)\cap G\ne\emptyset$ whenever $(a, b)\in A$ means exactly that
$A$ is covered by $L_X(G)$, where  $L_X(G)$ is defined as follows:
$$L_X(G)=\cup\{L_X(z): z\in G\}, \text{where 
$L_X(z)=\{(z-bx, b): b\in\re,\ x\in X\}$ for any $z\in\re.$}$$
Informally, to create  $L_X(G)$ we take for each $z \in G$ (thought as the point $(z, 0)$ in the plane) the lines 
$\{(z-bx, b): b\in \re \}$ through that point, one for each $x \in X$; the set $X$ determines slopes of these lines in apparent way.


The proof of non-universality of such a set $X$ can be equivalently reformulated as  the following plane covering problem:
prove that for any $\varepsilon>0$  there is a measurable (or, equivalently, open) set $G$ with $\lambda(G)<\varepsilon$ such that $L_X(G)$ covers all of $\re\times(\re\setminus\{0\})$.

{\bf Localization.} Inherent translation and scale invariance of the problem implies that for to solve the above mentioned covering problem it is sufficient to find  locally covering by $L_X(G)$ with arbitrarily small  $\lambda(G)$, or, equivalently, for a single set $A \subseteq \re\times(\re\setminus\{0\})$ with nonempty interior; 
our canonical choice will be $A=[0,1] \times [1,2].$
A set $X$ is non-universal iff
for any $\varepsilon>0$  there is a measurable (or, equivalently, open) set $G$ with $\lambda(G)<\varepsilon$ such that $L_X(G)$ covers 
$[0,1] \times [1,2]$.

{\bf Finite patterns reduction.} As $G$ above can be taken open, if $L_X(G)$ covers the square $[0,1] \times [1,2]$ then, by a simple compactness argument, this square is covered also by $L_Y(H)$ for some finite set $Y \subseteq X$ and for a set $H$ consisting of finitely many components of $G$.
Hence one of our main tasks here is to find, for any given $\varepsilon>0$,
the conditions on a finite set $Y \subset \re$ that would be sufficient for existence
of a set $G$ (consisting of finitely many intervals, say) with $\lambda(G)<\varepsilon$ such that $[0,1] \times [1, 2] \subseteq L_Y(G)$.

For such a finite set $Y$ and for any sufficiently short open interval $(a, b)$,
the set $L_Y((a, b))$ when restricted to the strip $\re \times [1,2]$ is a "bush-shape" plane set; we will study to some detail covering properties of the families of such sets.

While an original Erd\"os question was only interesting for $X$ bounded, the question of whether a given infinite set $X$ has the property that there are measurable sets
$C\subset \re$ with the complement of arbitrarily small (positive) measure containing
no similar copy of~$X$ seems to be for unbounded set equally interesting as for
bounded one.  For example, any sequence $\{k^{\alpha }\}^\infty_{k=1}$ for $\alpha
\in(0, 1)$ possess this property but, on the other hand, the set
$\rn=\{k\}^\infty_{k=1}$ of positive integers does not, as any set $C\subseteq \re$
with $\lambda (\re\setminus C)<\infty$ contains plenty of similar copies of~$\rn$.

\begin{figure}
\begin{tikzpicture}[yscale=0.4,xscale=0.6]
\node (A) at (12,0) {};
\coordinate (B) at (13,0);
\draw[->] (0,0) -- (23,0);
\draw (0,8) -- (22,8);
\draw (0,16) -- (22,16);
\draw[->] (0,0) -- (0,19);
\node at (23,-0.5) {$a$};
\node at (23,8) {$b=1$};
\node at (23,16) {$b=2$};
\node at (-0.8,19) {$b$};

\draw  (10,0) to (1.5,17);
\draw  (11,0) to (2.5,17);
\fill[gray] (10,0) -- (11,0) -- (2.5,17) -- (1.5,17);
\draw  (10,0) to (5,17);
\draw  (11,0) to (6,17);
\fill[gray] (10,0) -- (11,0) -- (6,17) -- (5,17);
\draw  (10,0) to (9,17);
\draw  (11,0) to (10,17);
\fill[gray] (10,0) -- (11,0) -- (10,17) -- (9,17);
\draw  (10,0) to (14,17);
\draw  (11,0) to (15,17);
\fill[gray] (10,0) -- (11,0) -- (15,17) -- (14,17);
\draw  (10,0) to (17,17);
\draw  (11,0) to (18,17);
\fill[gray] (10,0) -- (11,0) -- (18,17) -- (17,17);
\draw [line width=10pt] (10,0)--(11,0);
\end{tikzpicture}

\caption{A plane covering problem using "bush-shape" sets}\label{fig1}
\end{figure}
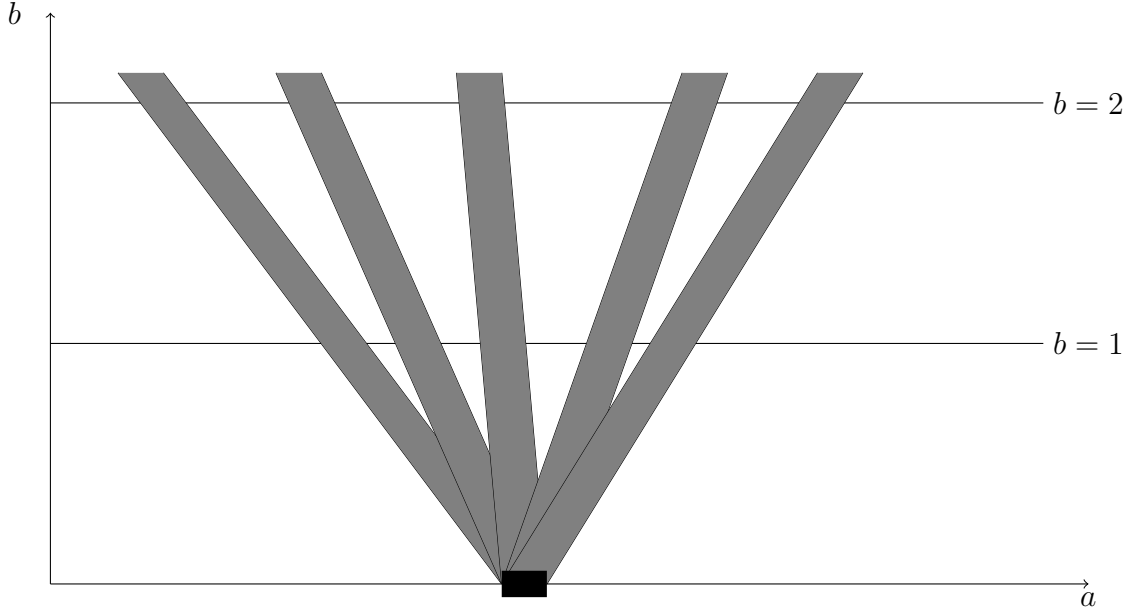

\medskip
{\bf Notation.} By $\lambda $ and $\lambda ^k$ we denote the outer \lm\ in $\re$ and $\re^k$,
respectively. For $E\subseteq \re^2$ and $b\in\re$ let $E^b$ stand for $\{a\in\re:(a,
b)\in E\}$.

Let $X\subseteq \mathbb R$ be an arbitrary set. Put 
\begin{align*}
L_X(z)&=\{(z-bx, b): b\in\re,\ x\in X\} \text{ for any $z\in\re$, and }\\
L_X(G)&=\cup\{L_X(z): z\in G\} \text{  whenever  $G\subseteq \re$.}
\end{align*}

Let us denote by $\mu_X$ a set function on $\re \times (\re\setminus\{0\})$ defined by
the formula
$$\mu_X(A)=\inf\{\lambda(G):G\subseteq \re\text{\ and \ }L_X(G)\supseteq
A\}\text{\  for any \  } A\subseteq \re \times (\re\setminus\{0\})\,.$$
If $X$ is infinite, define a set function $\tilde \mu_X$ on $\re \times (\re\setminus\{0\})$
 by the formula
$$\tilde\mu_X(A)=\sup\{\mu_{X\setminus K}(A):K\subseteq X\text{\ is finite$\}$\
for any\ }A\subseteq \re \times (\re\setminus\{0\}).$$
For $\delta \in (\re\setminus\{0\})$
 let $\varphi_\delta$ stand for the linear mapping $(a, b)\mapsto (\delta a, b)$.\\ 
 For $c \in\re$ let $\psi_c$ stand for the linear mapping $(a, b)\mapsto (a +bc, b)$.

\begin{proposition}\label{first}
Let $X\subseteq \re$ be an arbitrary set. Then the following hold:
\begin{enumerate}
\item [(i)] $(a, b)\in L_X(z)$ iff $z\in(a+bX)$, whenever $(a, b)\in\re^2$ and $z\in \re$.
\item [(ii)] $(a, b)\in L_X(G)$ iff $(a+bX)\cap G\ne\emptyset$, whenever $(a, b)\in\re^2$
and $G\subseteq\re$.
\item [(iii)] $L_X(\alpha +\beta G)= (\alpha, 0) + \beta L_X(G)$ whenever $\alpha , \beta
\in\re$ and $\beta \ne 0$.
\item [(iv)] $L_{\delta X}(G)=\varphi_\delta [L_X(\delta ^{-1} G)]$ whenever $\delta
\in\re\setminus\{0\}$.
\item [(v)] $L_{X-c}(G)=\psi_c[L_X(G)]$ whenever $c
\in\re$.

\item [(vi)] $L_{\overline X} (G)=L_X(G)$ whenever $G$ is open.


\item [(vii)] $\mu_X$ is an outer measure on $\re \times (\re\setminus\{0\})$.
\item [(viii)] For each $A\subseteq \re \times (\re\setminus\{0\})$
there are open sets $G_k\subseteq \re$ ($k=1, 2, \dots$) such that
$L_X(G_k)\supseteq A$ and $\lambda (G_k)\le \mu_X(A)+2^{-k}$.
\item [(ix)] $\mu_X\leq \mu_Y$ whenever $Y\subseteq X$ and $\mu_{\overline X}=\mu_X$.
\item [(x)] $\mu_X[(\alpha, 0) +\beta A]=|\beta |\mu_X(A)$ whenever $\alpha , \beta
\in\re$, $\beta \ne 0$ and $A\subseteq \re \times (\re\setminus\{0\})$.
\item [(xi)] $\mu_{\delta ^{-1}X}(A)=|\delta |^{-1}\mu_X[\varphi_\delta (A)]$ whenever
$\delta \in\re\setminus\{0\}$ and $A\subseteq \re \times (\re\setminus\{0\})$.
\item [(xii)] For any $A\subseteq \re \times (\re\setminus\{0\})$ the following are
equivalent:

\ods{(a)} $\mu_X(A)=0$\par
\ods{(b)} For each $\varepsilon >0$ there exists an open set $G\subseteq \re$ with
$\lambda (G)<\varepsilon $ such that 

$(a+bX)\cap G\ne\emptyset$ whenever $(a, b)\in A$.
\item [(xiii)] If $X\subseteq\re$ is infinite then (vi)--(xi) hold with $\tilde\mu_X$
instead of $\mu_X$.
\item [(xiv)] If $X\subseteq\re$ is infinite then for any $A\subseteq \re \times (\re\setminus\{0\})$
the following are equivalent:

\ods{(a)} $\tilde\mu_X(A)=0$\par
\ods{(b)} For each $\varepsilon >0$ there exists an open set $G\subseteq \re$ with
$\lambda(G)<\varepsilon $ such that 

$(a +bX)\cap G$ is infinite whenever $(a, b)\in A$.
\end{enumerate}
\end{proposition}

\begin{proof} (i) and (ii) are obvious.

\smallskip
(iii) $L_X(\alpha +\beta G)=\{(\alpha +\beta y-bx, b): b\in\re, x\in X, y\in G\}=
(\alpha, 0)+\beta \{(y-\beta ^{-1}bx, \beta ^{-1}b):b\in\re, x\in X, y\in G\}=
(\alpha, 0) +\beta  L_X(G)$.

\smallskip
(iv) $L_{\delta X} (G)=\{(z-b\delta x, b):b\in\re, x\in X, z\in G\}=\varphi_\delta \{(
\delta ^{-1}z-bx): b\in\re, x\in X, z\in G\}=\varphi_\delta [L_X(\delta ^{-1}G)]$.

\smallskip
(v) $L_{X-c} (G)=\{(z-b(x-c), b): b\in\re, x\in X, z\in G\}=\psi_c\{(z-bx, b): 
b\in\re, x\in X, z\in G\}=\psi_c [L_X(G)]$.

\smallskip
(vi) If $G$ is open, $(a +bX)\cap G\ne \emptyset$ iff $(a+b\overline X)\cap
G\ne\emptyset$ and we have $L_X(G)=L_{\overline
X}(G)$ due to (ii) .

\smallskip
(vii) Let $A, A_1, A_2, \dots \subseteq \re \times (\re\setminus\{0\})$ be such that 
$A\subseteq \cup^\infty_{i=1}A_i$.\newline
To prove $\mu_X(A)\le \sum^\infty_{i=1}\mu_X(A_i)$ we keep any $\varepsilon >0$ fixed 
and choose
$G_i\subseteq \re$ ($i=1, 2, \dots$) such that $L_X(G_i)\supseteq A_i$ and $\lambda
(G_i)\leq \mu_X(A_i)+2^{-i}\varepsilon $. Let $G=\sum^\infty_{i=1}G_i$. Then
$L_X(G)\supseteq \sum^\infty_{i=1} L_X(G_i)\supseteq \cup^\infty_{i=1} A_i\supseteq
A$ and $\lambda (G)\leq \sum^\infty_{i=1}\lambda (G_i)\leq \varepsilon
+\sum^\infty_{i=1}\mu_X(A_i)$. From that we get
$$\mu_X(A)\leq \lambda (G)\leq \varepsilon
+\sum^\infty_{i=1}\mu_X(A_i)\text{\ \ for any $\varepsilon >0$},$$
hence $\mu_X(A)\leq \sum^\infty_{i=1}\mu_X(A_i)$ and $\mu_X$ is an outer measure.

\smallskip
(viii) For each $A\subseteq \re \times (\re\setminus\{0\})$ we can choose $H_k\subseteq \re$ such that
$L_X(H_k)\supseteq A$ and $\lambda (H_k)\leq \mu_X(A)+2^{-k-1}$. Using the regularity
property of Lebesgue outer measure we can find open sets $G_k\supseteq H_k$ such that
$\lambda (G_k)\leq (H_k)+2^{-k-1}\leq \mu_X(A)+2^{-k}$.

\smallskip
(ix) Obviously $\mu_X\leq \mu_Y$ if $Y\subset X$. $\mu_{\overline X}=\mu_X$ in
virtue of (vi) and (viii).

\smallskip
(x) Using (iii) we obtain
$\mu_X[(\alpha,0) +\beta A]=\inf\{\lambda (G):L_X(G)\supset [(\alpha, 0) +\beta A]\}=\\
\inf\{\lambda [\alpha H+(0, \beta )]:L_X[\alpha H+(0, \beta )]\supseteq [\alpha
A+(0, \beta 0]\}=|\alpha |\inf\{\lambda (H):L_X(H)\supseteq A\}=|\alpha |\mu_X(H).$

\smallskip
(xi) Using (iv) we obtain
$
\mu_{\delta ^{-1} X}(A)=\inf\{\lambda (G):L_{\delta ^{-1}X}(G)\supseteq
A\}=\inf\{\lambda (G):L_X(\delta G)\supseteq\varphi_\delta (A)\}=\inf\{\lambda (\delta ^{-1}H):L_X(H)\supseteq \varphi_\delta (A)\}=|\delta
|^{-1}\mu_X[\varphi_\delta (A)].
$

\smallskip
(xii) By definition, $\mu_X(A)=0$ iff $\forall\,\varepsilon >0\,\exists\,G\subseteq
\re$ such that $\lambda (G)<\varepsilon $ and $L_X(G)\supseteq A$.
In virtue of (viii) $G$ can be chosen to be open and by (ii)
$L_X(G)\supseteq A$ iff
$(a+bX)\cap G\ne \emptyset$ whenever $(a, b)\in A$.

\smallskip
(xiii) It now easily follows given how $\tilde\mu_X$ is defined using 
$\mu_{X\setminus K}$.

\smallskip
(xiv)  (b) $\implies$ ($\mu_{X\setminus K}(A)=0$ for any $K\subseteq X$ finite)
$\implies$ (a)\newline
To prove (a) $\implies$ (b) we take a countable dense subset 
 $X_1=\{x_i\}^\infty_{i=1}$ of $X$, and  put
$X_k=\{x_i\}^\infty_{i=k}$ for any $k\in\rn$.\newline
(a) $\implies$ ($\mu_{X_k}(A)=0$ for any $k\in\rn$) $\implies$ (for any $k\in\rn$ and any
$\varepsilon >0$ there exist open sets $G_k\subseteq \re$ with $\lambda
(G_k)<2^{-k}\varepsilon $ and $(a+bX_k)\cap G_k\ne \emptyset$ whenever $(a, b)\in A$
$\implies$ (b), putting $G=\cup^\infty_{k=1}G_k$.
\end{proof}

\begin{theorem}\label{two} For any set $X\subseteq \re$ the following conditions are equivalent:
\begin{enumerate}
\item [(i)] For every $\varepsilon >0$ there exists a closed set $C\subseteq \re$ with
$\lambda (\re\setminus C)<\varepsilon $ containing no similar copy of~$X$.
\item [(ii)] $\mu_X\equiv0$ on $\re \times (\re\setminus\{0\})$.
\item [(iii)] There exists a set $U\subseteq \re \times (\re\setminus\{0\})$
with nonempty interior such that
$\mu_X(U)=0$.
\end{enumerate}
\end{theorem}

\begin{proof}

(i) is equivalent to (ii) by Proposition~\ref{first}(xii), where we put $C=\re\setminus G$.
\smallskip
(ii) $\implies$ (iii) is obvious. To prove an opposite implication choose $\alpha _i$, $\beta _i$ ($i=1, 2,
\dots$) such that $\cup^\infty_{i=1}[(\alpha _i, 0) +\beta _i U]=\re \times (\re\setminus\{0\})$
and use
Proposition~\ref{first}(x).
\end{proof}

\section{Deterministic locally periodic layering}

There are various ways how we can try to prove that  certain infinite patterns $X \subseteq \re$ are non-universal. One of many equivalent ways how we can restate this question is the following plane covering problem:  A set $X\subseteq \re$ is non-universal iff
for any $\varepsilon>0$  there is a measurable set $G\subseteq \re$ with $\lambda(G)<\varepsilon$ such that $L_X(G)$ covers $[0,1] \times [1,2]$. Given  $\varepsilon>0$ one can prove that such a set $G$ exists using either deterministic or probabilistic constructions. In this section we will describe how the presence of large patterns in $X$ that are {\it relatively fine} allows to prove in deterministic way that small sets $G$ with $L_X(G)$ covering $[0,1] \times [1,2]$ exist. We will show that if the set $X \subseteq \re$ is {\it arbitrarily relatively fine} (the notion introduced below) then  there is a simple way how to solve our local plane covering problem with $G$ consisting of a finite periodic collection of intervals.

{\bf Sets that are  $\varepsilon$-fine; $\delta$-separated sets.} Let $Y=\{y_i\}^k_{i=1}$ ($k\geq 2$) be real numbers such that $y_1>y_2>\dots  > y_k$.
We call such a set {\it  $\varepsilon$-fine} if $y_i-y_{i+1} \leq  \varepsilon$ for each $i=1, 2, \dots, k-1$.
Similarly, it is called {\it $\delta$-separated} if $y_i-y_{i+1} \geq \delta$ for each $i=1, 2, \dots, k-1$.

In this scale invariant problem it is useful to measure quality of how sets are fine and/or  separated relatively to the size of an interval they occupy. 

{\bf Relatively  $\varepsilon$-fine sets, relatively $\delta$-separated sets.} We call such a set $Y$ 
{\it relatively  $\varepsilon$-fine} if $\frac{y_i-y_{i+1}}{y_1-y_k} \leq  \varepsilon$ for each $i=1, 2, \dots, k-1,$ and
{\it relatively $\delta$-separated} if $\frac{y_i-y_{i+1}}{y_1-y_k} \geq \delta$ for each $i=1, 2, \dots, k-1$.

{\bf Arbitrarily relatively fine sets.} We call a set $X\subseteq \re$ {\it arbitrarily relatively fine} if it has the following property:\\
For every $\varepsilon >0$ there exist elements $y_1>y_2>\dots>y_k$  ($k\geq 2$) in~$X$
such that
$$(1+|y_1|+|y_k|)(y_i-y_{i+1})\leq \varepsilon(y_1-y_k) \text{\ \ for any\ \ }i=1, 2,
\dots, k-1.
$$
For a bounded set $X$ this condition means exactly that for arbitrarily small  $\varepsilon > 0$ there exists
a relatively  $\varepsilon$-fine subset of $X$; for $X$ unbounded it is additionally required that such relatively  $\varepsilon$-fine subset of $X$ can be found  in $\{x \in X: |x|\leq o(\frac{1}{ \varepsilon}) \}$, if  $\varepsilon$ tends to $0$.

As the property of being {\it arbitrarily relatively fine} is described by a limit condition, it is easy to see that  it is stable under removing finitely many elements. Namely, if $X\subseteq \re$ is  {\it arbitrarily relatively fine}, then 
$X \setminus K$ is  as well  for any finite set $K$.

One can equivalently describe the property of $X$ being {\it arbitrarily relatively fine} by
$$\inf\{\frac{X(u, v)}{v-u}(|u|+|v|+1): u,v \in\re, u<v \}=0\,,$$
where $X(u, v)$ denotes the length of the longest component of $(u, v) \setminus \overline X$.

\begin{lemma}\label{three} Let $Y=\{y_i\}^k_{i=1}$ ($k\geq 2$) be real numbers such that $y_1>y_2>\dots y_k\geq 0$
and put $M=\max\{y_i-y_{i+1}:i= 1, \dots, k-1\}$. Then $\mu_Y([0, 1] \times [1, 2])\leq 4M\frac{1+y_1}{y_1-y_k}$.
\end{lemma}

\begin{proof} Let us denote $r=\left\lceil \frac{1+y_1}{2M+y_1-y_k}\right\rceil$,
$u_j=y_1+j(2M+y_1-y_k)$ for any integer $j$ and $G=\cup^r_{j=0}[u_j, u_{j}+2M]$. Then
$u_{r+1}-2y_1>1$ and obviously $r<\frac{1+y_1}{y_1-y_k}$.

As $\lambda (G)=2M(r+1)<4M\frac{1+y_1}{y_1-y_k}$, it is sufficient to prove that
$L_Y(G)\supseteq [0, 1]\times [1, 2]$.

We will show that for any $b\in[1, 2]$ we have
$$L_Y([u_j, u_j+2M])\supseteq [u_j-by_1, u_j+2M-by_k]\times \{b\} \supseteq 
[u_j-by_1, u_{j+1}-by_1] \times \{b\}$$
and, consequently, $L_Y(G)\supseteq [u_0-by_1, u_{r+1}-by_1] \times \{b\}
\supseteq [0, 1] \times \{b\}$.

To prove $L_Y([u, u+2M])\supseteq [u-by_1, u+2M-by_k] \times \{b\}$ whenever $b\in [0,
2]$ and $u$ is arbitrary, keep any real $u$
and $b\in [0, 2]$ fixed, and also $a\in [u-by_1, u+2M-by_k]$.

Let $N=\max\{i:1\leq i\leq k\text{\ and\ } a\geq u-by_i\}$. Then we see that
$u-by_N\leq a\leq u-by_N+2M$. 

\noindent Put $v=by_N+a$. Then $v\in[u, u+2M]$, $v-by_N=a$,
hence $(a, b)\in L_Y(v)\subseteq L_Y([u, u+2M])$.\newline
Thus the first inclusion above is proved and the second one holds for $b\geq 1$ by
definition of~$u_j$.
\end{proof}

\begin{theorem}\label{four} Let  $X\subseteq \re$ be a set that is  {\it arbitrarily relatively fine}, then $X$ is non-universal (modulo similarities) w.r.t.\ the sets with complement of finite Lebesgue measure.  For every $\varepsilon >0$ there exists an open set $G\subseteq \re$ with
$\lambda (G)<\varepsilon $ such that any similar copy of $X$ intersects $G$ in an
infinite set. In another words, the set $ \re \setminus G$ doesn't contain any similar copy of $X \setminus K$, for any finite set~$K$.
\end{theorem}

\begin{proof}
(a) Let us prove first that if $X$ is contained in $(0, \infty)$ and it is  {\it arbitrarily relatively fine}
then $\mu_X\equiv 0$. Using Lemma~\ref{three} under these assumptions we easily obtain that $\mu_X([0, 1]\times [1, 2])=0$ and hence
$\mu_X\equiv 0$ by Theorem~\ref{two}.

(b) If $X\subseteq \re$ is any set  that is {\it arbitrarily relatively fine}, then it is easy to
observe that there exists a countable {\it arbitrarily relatively fine} subset $X_1=\{x_i\}_{i=1}^{\infty}$ 
of $X$ such that either $X_1\subseteq (0, \infty)$, 
or $X_1\subseteq (-\infty, 0)$. We also denote $X_k=\{x_i\}^\infty_{i=k}$ for
any $k=2, 3, \dots$. It is easy to see that each $X_k$ is {\it arbitrarily relatively fine} as well.
By part (a) (and using Proposition~\ref{first}(xii) in case $X_1\subseteq (-\infty, 0)$) we
obtain that $\mu_{X_k}\equiv 0$ for any $k=1, 2, \dots$, hence $\tilde
\mu_{X_1}\equiv 0$. Due to Proposition~\ref{first}(xiv) for every $\varepsilon >0$ there exists an open set
$G\subseteq \re$ with $\lambda (G)<\varepsilon $ such that any similar copy of $X_1$
intersects $G$ in an infinite set.
\end{proof}

\begin{example}[bounded]\label{five}
	 It is easy to see that any zero-sequence that converges to $0$ so
slowly that $\lim_{i\to \infty}\frac{x_{i+1}}{x_i}=1$ (as assumed in  Falconer's result) is {\it arbitrarily relatively fine}.
But a zero sequence may be {\it arbitrarily relatively fine} even if it decreases at time much faster if it slows down accordingly from time to time, as in the following example.

We choose first a subsequence $\{x_{k^2}\}^\infty_{k=1}$ convergent to $0$ as fast as we wish,
and then choose $0< \varepsilon_k < \frac{x_{(k+1)^2}-x_{k^2}}{4k^2}$. We will introduce the other elements of the sequence in such a way that, for every $k$, $X_k=\{x_{k^2+i}=x_{k^2}+i\varepsilon_k:
i=0,1,\dots, 2k\}$ is a $(2k+1)$-tuple that is $\frac{1}{2k}$-fine. Of course, the patterns that are relatively fine needn't be so regular, we used equally spaced patterns above for convenience only.
\end{example}

\begin{example}[unbounded]
Let $\{x_k\}^\infty_{k=1}$ be an increasing sequence of
positive numbers tending to $+\infty$ such that for every $\varepsilon >0$ there
exist integers $1\leq m<n$ such that $x_{k+1}-x_k <\varepsilon
\left(1-\frac{x_m}{x_n}\right)$ for any $k=m, m+1, \dots, n-1$. Then $X$ is {\it arbitrarily relatively fine}.
This is implied by $x_k\to+\infty$ and $(x_{k+1}-x_k)\to0$, for example, but it can be achieved also in sequences with many large gaps among $x_{k+1}-x_k$ , as in the following example.

We choose first a sequence $\{y_k\}^\infty_{k=1}$ that tends to $+\infty$ arbitrarily fast; it will be a subsequence of our example. For every $k$, in a short left neighborhood of $y_k$ we introduce
a pattern of new $\omega(y_k)$ equally spaced elements, so that it will be $o(\frac{1}{y_k})$-fine.
Any sequence created this way will be  {\it arbitrarily relatively fine}.
\end{example}

\begin{remark} Any set $X$ containing the patterns that are relatively fine, then has to contain the  patterns that are both, fine and well separated. In particular, it is easy to see that any set $Y=\{y_i\}^k_{i=1}$ ($k\geq 2$) of real numbers with $y_1>y_2>\dots y_k$ that is {\it  relatively $\varepsilon$-fine} contains a subsequence (with the first term $y_1$ and the last one $y_k$) that is both, relatively $\varepsilon$-separated and  relatively $3\varepsilon$-fine. Hence arbitrarily relatively fine sets contain, for arbitrarily small $\varepsilon >0$, relatively$\varepsilon$-separated and  relatively $3\varepsilon$-fine sets; the cardinality of such sets
is bounded from above by $\frac{1}{\varepsilon}+1$ and from below by $\frac{1}{3\varepsilon}+1$.
We have seen by the deterministic construction that presence of such patterns for arbitrarily small $\varepsilon >0$ in a bounded set X implies non-universality of $X$. In Section~\ref{sec-five} we will prove using probabilistic methods that the presence in $X$ of much smaller relatively $\varepsilon$-separated subsets (with arbitrarily small $\varepsilon>0$) is already sufficient for the proof of non-universality, namely the cardinality $\omega(|\ln{\varepsilon}|)$ of relatively $\varepsilon$-separated  subsets  is sufficient.
\end{remark}

\section{Avoiding almost all similar copies}

One can observe from the proof of universality of a finite set $X \subseteq \re$ that for a Lebesgue measurable set $C \subseteq \re$ of positive measure there are many similarity mappings $(a+bX)$ for which $(a+bX) \subseteq C$, namely $\lambda^2 \{(a, b) \in \rmr:  (a+bX) \subseteq C\}>0$.
So it is natural to consider also universality of sets in such stronger sense;  $X$ being {\it  strongly universal (modulo similarities) w.r.t.\ the sets with complement of finite Lebesgue measure} if  
$\lambda^2 \{(a, b) \in \rmr:  (a+bX) \subseteq \re \setminus G\}>0$ whenever $\lambda(G)<\infty$.
This will lead to a somewhat relaxed notion of non-universality;  $X$ is not {\it  strongly universal (modulo similarities) w.r.t.\ the sets with complement of finite Lebesgue measure} if  for every $\varepsilon >0$ there exists an open set $G\subseteq \re$ with $\lambda (G)<\varepsilon $ such that
$\lambda^2 \{(a, b) \in \rmr:  (a+bX) \subseteq \re \setminus G\}=0$. 
An original Erd\"os problem to prove non-universality of any bounded infinite set is relaxed to a
problem that is much easier. One can find small sets $G$ for which $\re \setminus G$ avoids almost all similar copies of~$X$. In terms of our reformulation as a plane covering problem,
one aims to cover almost everything, rather than everything. In this section we will study covering properties of sets $L_X(G)$. Our geometric method allows to deal with rather general Borel measures in the plane, not merely the Lebesgue measure.As we consider unbounded sets $X$ as well, it is an interesting question to understand which infinite sets are strongly universal. The following notion will play an important role.

\medskip

{\bf Linearly bounded locally finite sets.} We call a set $X \subseteq \re$ {\it a linearly bounded locally finite set} if
$$\sup\{n^{-1}\card(X\cap[-n, n]):n\in\rn\}<\infty,$$
or, equivalently, $\card(X\cap[-u, u]) \leq Mu$ for a constant $M$ and each $u \geq1$.

\begin{theorem}  Let $X\subseteq \re$. We can observe the following dichotomy concerning strong universality/non-universality modulo similarities w.r.t.\ the sets with complement of finite measure.

\begin{enumerate}
\item [(a)] If $X$ is a linearly bounded locally finite set then it is strongly universal modulo similarities w.r.t. the sets with complement of finite measure. In particular, there is a constant $\varepsilon>0$ such that whenever $G\subseteq \re$ is a Lebesgue measurable set with  $\lambda(G)<\varepsilon$, then the set $\{(a, b) \in \re\times [1, 2]: (a+bX) \subseteq (\re \setminus G)\}$ of restricted scale similarities that put $X$ into $(\re \setminus G)$ has infinite $2$-dimensional Lebesgue measure.

\item [(b)] If $X$ is not a linearly bounded locally finite set then it is not strongly universal modulo similarities w.r.t.\ the sets with complement of finite measure. For every $\sigma$-finite Borel measure $\nu$ in $\rmr$ and for every $\varepsilon >0$ there exists an open set $G\subseteq \re$ with $\lambda (G)<\varepsilon $ such that $(a+bX)\cap G$ is infinite for $\nu$ almost every $(a, b)\in\rmr$.
\end{enumerate}
\end{theorem}

\begin{proof}
	
(a)  Let $X$ be a linearly bounded  locally finite set, and $M$ be a constant such that $\card(X\cap [-u, u]) \leq 
Mu\,$ whenever $u \geq1$. We will show first that  there is a constant $K$ such that $\lambda^2(L_X(G) \cap ([-1, 1]\times [1, 2])) \leq K \lambda(G)$ whenever $G\subseteq \re$. As a constant $K \geq 2$ will do, we need to consider sets $G$ with $\lambda(G) \leq 1$ only. Any set $G\subseteq \re$ split into its subsets $G_0=G \cap [-3, 3]$, $G_+ = G \cap (3, +\infty)$, $G_- = G \cap (- \infty, -3)$. Let's estimate from above the measure $\lambda^2(L_X(G_0) \cap ([-1,1]\times [1, 2]))$ first. One can check that the set of those $x\in X$ for which the set $L_{\{x\}}([-3, 3])$ intersects the rectangle $[-1,1]\times [1, 2]$  is exactly $X\cap [-4, 4]$, so its cardinality is at most $4L$ and we easily get that $\lambda^2(L_X(G_0) \cap ([-1,1]\times [1, 2])) \leq 4M \lambda(G_0)$.

Now we estimate  $\lambda^2(L_X(G_+) \cap ([-1,1]\times [1, 2]))$ from above in terms of $ \lambda(G_+)$.
(We can assume that $G$ is open; otherwise we take its open superset of slightly larger measure.) Consider all components $G_i=(d_i, d_i+\varepsilon_i)$, $i=1, 2, \dots$, of the set $G_+$, one by one. Remind that $d_i \geq 3$ then, and that we can confine to $\varepsilon_i \leq 1$.
The set of those $x\in X$ for which the set $L_{\{x\}}(G_i)$ intersects the rectangle $[-1,1]\times [1, 2]$  is exactly $X\cap (\frac{d_i-1}{2}, d_i+1+\varepsilon_i)$. While its cardinality, which is at most $M(d_i+1+\varepsilon_i) \leq 2Md_i$, can be arbitrarily large with increasing $d_i$, we will show that 
$\lambda^2(L_X((d_i, d_i+\varepsilon_i)) \cap ([-1,1]\times [1, 2]))/ \varepsilon_i$ can be bounded from above independently of how large $d_i$ can be. This is because the width of the strip $L_{\{x\}}((d_i, d_i+\varepsilon_i))$ is small for $x$  large, its width in horizontal direction is $\frac{\varepsilon_i}{x}$, and the area in which it intersects $[-1,1]\times [1, 2]$ is at most
 $\lambda^2(L_{\{x\}}((d_i, d_i+\varepsilon_i)) \cap ([-1,1]\times [1, 2])) \leq 
\frac{3 \varepsilon_i}{x}$.

For each $x$ belonging to $X\cap (\frac{d_i-1}{2}, d_i+1+\varepsilon_i) \subseteq X\cap (\frac{d_i}{4}, 2d_i)$ this area will be at most $\frac{12 \varepsilon_i}{d_i}$, and as the number of them is at most $2Md_i$ we get  $\lambda^2( L_X((d_i, d_i+\varepsilon_i)) \cap ([-1,1]\times [1, 2])) \leq 
24M \varepsilon_i$.

After summing over all components of $G_+$ we get $\lambda^2( L_X(G_+) \cap ([-1,1]\times [1, 2])) \leq 24M \lambda(G_+)$.

By symmetry we can get the same estimate with $G_-$ in place of $G_+$. Taking $K= \max(24M, 2)$ we finally get
$\lambda^2(L_X(G) \cap ([-1, 1]\times [1, 2])) \leq K \lambda(G)$ whenever $G\subseteq \re$, as needed.

So if $G\subseteq \re$ is a Lebesgue measurable set with  $\lambda(G)<\frac{2}{K}$, then the set $\{(a, b) \in [-1, 1] \times [1, 2]: (a+bX) \subseteq (\re \setminus G)\}$, which is exactly 
$([-1, 1]\times [1, 2]) \setminus L_X(G)$ has its measure
$\lambda^2 \{(a, b) \in [-1, 1] \times [1, 2]: (a+bX) \subseteq (\re \setminus G)\} \geq 2-K \lambda(G)>0$.
As the problem is translation invariant in $a$-variable, the same measure estimate will apply to any rectangle $[u, u+2] \times [1, 2]$ in place of $[-1, 1] \times [1, 2]$, hence
$\lambda^2 \{(a, b) \in \re \times [1, 2]: (a+bX) \subseteq (\re \setminus G)\}=\infty$.  By scale invariance
of the problem we can conclude universality of $X$ modulo similarities w.r.t. the sets with complement of finite measure (and not only w.r.t. the sets with complement of sufficiently small measure, as stated in this proof).

That completes the proof of part (a).

\smallskip

(b) The rest of this section is devoted to developing techniques leading to the proof of part (b)
that is contained in Theorem~\ref{twelve}  below.
\end{proof}

\begin{lemma}\label{eight} Let $E_1$, $E_2$, \dots, $E_k$ be Borel subsets of $\re^2$, $\nu
$ be a Radon measure in $\re^2$ and $U=\prod^k_{i=1} [A_i, B_i]$ be a bounded
interval in $\re^k$. Then there exists $u=(u_1, u_2, \dots, u_k)\in U$ such that
$$\nu\left(\re^2\setminus \bigcup^k_{i=1}\left((u_i, 0)+E_i\right)\right)\leq
\int_{\re^2} e^{-\sum^k_{i=1}\left(B_i-A_i\right)^{-1}\lambda\left(E^b_i\cap [a-B_i,
a-A_i]\right)} \\d\nu ((a,b)).$$
\end{lemma}

\begin{proof}
For a fixed $u\in U$ let us denote by $F_u$ the characteristic function of the set
$\re^2\setminus\cup^k_{i=1}((u_i, 0)+E_i)$. Obviously, $F_u((a,
b))=\prod^k_{i=1}(1-\chi_{E_i}((a-u_i, b)))$. Put $J=\left[ \lambda
^k(U)\right]^{-1}\int_U\int F_u((a, b))\,d\nu d\lambda ^k(u)$, the mean value of the function
$u\mapsto \nu (\re^2\setminus \bigcup^k_{i=1}((u_i, 0)+E_i))$
over $u\in U$. 

Using Fubini theorem we obtain
\begin{align*}
J&=\left[\prod^k_{i=1}(B_i-A_i)\right]^{-1}\int^{B_1}_{A_1}\int^{B_2}_{A_2}\dots
\int^{B_k}_{A_k}\int_{\re^2}F_u((a, b))\,d\nu ((a, b))d\lambda (u_k)d\lambda
(u_{k-1})\dots d\lambda (u_1)\\
&=\int_{\re^2}\prod^k_{i=1}\left[(B_i-A_i)^{-1}\int^{B_i}_{A_i}(1-\chi_{E_i}((a-u_i, b))\,d\lambda (u_i)\right]\,d\nu ((a, b))\\
&=\int_{\re^2}\prod^k_{i=1}\left[1-(B_i-A_i)^{-1}\lambda (E^b_i\cap [a-B_i,
a-A_i])\right]\,d\nu ((a, b))\\
&\leq \int_{\re^2} e^{-\sum^k_{i=1}(B_i-A_i)^{-1}\lambda (E^b_i\cap [a-B_i,
a-A_i])}\,d\nu ((a, b)).
\end{align*}
As $J$ is the mean value of the function $u\mapsto \nu (\re^2\setminus \bigcup^k_{i=1}((u_i, 0)+E_i))$
over $U$, the existence of $u\in U$ satisfying the above inequality follows.
\end{proof}

\begin{lemma}\label{nine} Let $\nu$ be a Radon measure in $\re^2$ and
$\{E_i\}^\infty_{i=1}$ be a sequence of Borel subsets of $\re^2$ for which there
exist bounded intervals $[C_i, D_i]\subseteq \re$ such that
\begin{align} \sum^\infty_{i=1}(1+D_i-C_i)^{-1}\lambda (E^b_i\cap [C_i, D_i])=\infty\text{\ \ for
$\nu$-almost every $(a, b)\in\re^2$.}    \label{maxP}\end{align}


Then there exist $\{u_i\}^\infty_{i=1}\subseteq\re$ such that $\nu$-almost every
$x\in\re^2$ is in infinitely many of the sets $\{(u_i, 0)+E_i\}^\infty_{i=1}$.
\end{lemma}

\begin{proof}
(a) Let us first prove that under these assumptions for any bounded interval  $I\subseteq \re^2$ and for any
$\varepsilon >0$ there are $k\in\rn$ and $u_1, u_2, \dots, u_k\in\re$ such that
$$ \nu[I\setminus \bigcup^k_{i=1}((u_i, 0)+E_i)]<\varepsilon \,.$$
Choose $\varepsilon >0$, bounded interval $I=[\alpha , \beta ]\times[\gamma , \delta
]\subseteq \re^2$ and put $A_i=\alpha -D_i$, $B_i=\beta -C_i$, where $\{[C_i,
D_i]\}^\infty_{i=1}$ is a fixed sequence of real intervals for which~(\ref{maxP}) holds.

Then $[a-B_i, a-A_i]\supseteq[C_i, D_i]$ whenever $(a, b)\in I$, and due to~(\ref{maxP}) 
$$\sum^\infty_{i=1}(B_i-A_i)^{-1}\lambda (E^b_i\cap[a-B_i, a-A_i])=\infty$$
for $\nu$-almost every $(a, b)\in I$.

So there exists $k\in\rn$ such that
$$\int_I e^{-\sum^k_{i=1}(B_i-A_i)^{-1}\lambda (E_i^b\cap[a-B_i,
a-A_i])}\,d\nu(a, b)<\varepsilon \,.$$
Using Lemma~\ref{eight} for measure $\nu\llcorner I$ we obtain that there exists
$$(u_1, \dots, u_k)\in\prod^k_{i=1}[A_i, B_i]\text{\ \ such that\ \ }
\nu\left[I\setminus \bigcup^k_{i=1}((u_i, 0)+E_i)\right]<\varepsilon \,.$$
\smallskip
(b) Choose bounded intervals $I_j\subseteq \re^2$, $j=0, 1, 2, \dots$, such that
$I_0\subseteq I_1\subseteq I_2\subseteq\dots$ and $\cup^\infty_{j=0}I_j=\re^2$. Using
part (a) we obtain $k_1\in\rn$ and
$$u_1, u_2, \dots, u_{k_1}\in\re\text{\ \ such that\ \ }
\nu[I_0\setminus \bigcup^{k_1}_{i=1}((u_i, 0)+E_i)]<1\,.$$
Continuing in the same way we obtain by induction for each $j\in\rn$, a finite number
of real numbers $u_{k_j+1}, u_{k_j+2}, \dots, u_{k_{j+1}}$ such that
$$\nu\left[I_j\setminus\bigcup^{k_{j+1}}_{i=k_j+1}((u_i, 0)+E_i)\right]<2^{-j}\,.$$
Obviously $\nu(\re^2\setminus\cup^\infty_{i=k}((u_i, 0)+E_i))=0$ for every $k\in\rn$,
and the proof is complete.
\end{proof}

\begin{lemma}\label{ten} Let $X\subseteq\re$ be a set that is not linearly bounded, hence
$$\sup\{n^{-1}\card(X\cap[-n, n]):n\in\rn\}=\infty\,.$$
Then there are sequences $\{H_i\}^\infty_{i=1}$ and $\{[C_i, D_i]\}^\infty_{i=1}$ of
intervals in $\re$ with the following properties:
\begin{enumerate}
\item [(i)] $\sum^\infty_{i=1}\lambda (H_i)\leq 1$ and
\item [(ii)] $\sum^\infty_{i=1}(1+D_i-C_i)^{-1}\lambda ([L_X(H_i)]^b\cap[C_i,
D_i])=\infty$ for every $b\in[1, 2]$.
\end{enumerate}
\end{lemma}

\begin{proof}
For every $m\in\rn$ fix $n_m\in\rn$ such that $\card(X\cap[-n_m, n_m])\geq 2^mn_m$.

 Further choose an $\delta_m\in(0, 2^{-m})$ such that in $X\cap [-n_m, n_m]$
 there exists an $\delta _m$ separated set $Z_m$ with $2^mn_m$ elements.

Put $p_m=[2^{-m}\delta ^{-1}_m]$, $k_0=0$ and $k_m=\sum_{j=1}^mp_j$ for
$m\in\rn$. For $m\in\rn$ and any $i=k_{m-1}+1, k_{m-1}+2, \dots, k_{m-1}+p_m=k_m$
we take $$H_i=(0, \delta _m),\quad E_i=L_X(H_i)\quad\text{and}\quad [C_i,
D_i]=[-2n_m-1, 2n_m+1].$$
It is easy to check that the following hold for every $b\in[1, 2]$:

\medskip

\begin{enumerate}

\item[(a)] $E^b_i\supseteq[L_{Z_m}(H_i)]^b \supseteq (0, \delta _m)-bZ_m$,
 
\medskip

\item[(b)] $\lambda (E^b_i\cap[C_i, D_i])\geq 2^mn_m \delta _m$,

\medskip

\item[(c)] $\sum\{\lambda (H_i):k_{m-1}<i\leq k_m\}=p_m\delta _m\leq 2^{-m}$,

\medskip

\item[(d)] $\sum\{(1+D_i-C_i)^{-1}\lambda (E^b_i\cap[C_i, D_i]):k_{m-1}<i\leq
k_m\}\geq p_m(4n_m+3)^{-1}2^mn_m\delta _m\geq 2^{-4}$.

\medskip
\end{enumerate}

 So properties (i) and (ii) easily follow.
\end{proof}

\begin{proposition}\label{eleven} Let $X\subseteq \re$ be a set that is not  linearly bounded. If $\nu$ is a $\sigma$-finite Borel measure in $\rmr$ then there exists
a set $A\subseteq\rmr$ such that $\nu(\rmr\setminus A)=0$ and $\mu_X(A)=0$.
\end{proposition}

\begin{proof}
We can obviously assume that $\nu$ is finite. Due to (x) of Proposition~\ref{first} it is
sufficient to find $A$ such that $\nu((\re \times [1, 2])\setminus A)=0$ and
$\mu_X(A)=0$. Let $H_i$ and $[C_i, D_i]$ are such as in Lemma~\ref{ten}. Put $E_i=L_X(H_i)$
and apply Lemma~\ref{nine} with a measure $\nu\llcorner(\re \times [1, 2])$. We obtain a
sequence $\{u_i\}^\infty_{i=1}\subseteq\re$ such that $\nu$-almost every $x\in \re \times [1, 2]$ 
is in infinitely many of sets $\{(u_i, 0)+E_i\}^\infty_{i=1}$.

It follows that the set $A=\cap^\infty_{k=1}\cup^\infty_{i=k}L_X(u_i+H_i)$
satisfies $\nu((\re \times [1, 2])\setminus A)=0$.

Further, $\mu_X(A)\leq \sum^\infty_{i=k}\lambda (H_i)$ for any $k\in\rn$, so
$\mu_X(A)=0$.
\end{proof}

\begin{theorem}\label{twelve} Let $X\subseteq \re$ be a set that is not  linearly bounded,  and
$\nu$ be a $\sigma$-finite Borel measure in $\rmr$.

Then for every $\varepsilon >0$ there exists an open set $G\subseteq \re$ with
$\lambda (G)<\varepsilon $ such that $(a+bX)\cap G$ is infinite for $\nu$ almost
every $(a, b)\in\rmr$.
\end{theorem}

\begin{proof}
 We can choose a countable set $X_1=\{x_i\}^\infty_{i=1}\subset X$ that is not  linearly bounded.
 Obviously, any $X_k=\{x_i\}^\infty_{i=k}$ will possess the property of not  being linearly bounded
 as well. For any $k\in \rn$
 we can find by Proposition~11 a set $A_k$ of full $\nu$ measure such that
 $\mu_{X_k}(A_k)=0$. Then $A=\cap^\infty_{k=1}A_k$ is a set of full $\nu$ measure
 and $\tilde\mu_{X_1}(A)=0$. Hence by Proposition~\ref{first}(xiv) for every $\varepsilon >0$
 there exists an open set $G\subseteq \re$ with $\lambda (G)<\varepsilon $ such that
$(a + bX_1)\cap G$ is infinite for every $(a, b)\in A$.
\end{proof}

There are simple examples of measures $\nu$ for which if we can avoid almost all similar copies of $X$, then we can easily avoid all similar copies of $X$, as the following theorem shows.

\begin{theorem}\label{trinast} Let $X\subseteq \re$ be a set that is not  linearly bounded, and
$B\subseteq\rmr$ be a set of $\sigma$-finite $1$-dimensional Hausdorff measure. Then for every $\varepsilon >0$ there exists an open set $G\subseteq
\re$ with $\lambda (G)<\varepsilon $ such that $(a+bX)$ is infinite for every $(a,
b)\in B$.
\end{theorem}

\begin{proof}
Let $X_k$ have the same meaning as in proof of Theorem~\ref{twelve}. Using Proposition~11 we
can find a set $A_k$ with $\text H^1(B\setminus A_k)=0$ and $\mu_{X_k}(A_k)=0$. But any $\text H^1$-null set is also
$\mu_Y$-null, whenever $Y\ne\emptyset$. It follows that $\mu_{X_k}(B)=0$ and hence
$\tilde\mu_{X_1}(B)=0$. The proof is complete by Proposition~\ref{first}(xiv).
\end{proof}

\begin{remark}  To prove Theorem~\ref{trinast} for a set $B$ of the form $B=\re\times C$, where
$C\subseteq \re\setminus\{0\}$ is countable, we can relax the assumption that $X$ is not  linearly bounded to a weaker one that $X$ is not uniformly locally finite (i.e., $\sup\{\card(X\cap [u, u+1]):u\in\re\}=\infty\,$).This can be easily seen from the next section where the results about translation copies can be  extended in straightforward way to the case with countably many scales.
\end{remark}

\begin{remark}
 Our Theorem~\ref{twelve} is general and valid for every $\sigma$-finite Borel measure $\nu$ in $\rmr$,
but one can prove even stronger results for $\nu$ being the $2$-dimensional Lebesgue measure, or for more general product measures. The exceptional set of pairs 
$\{(a, b) \in \rmr: (a+bX) \subseteq (\re \setminus G)\}$  (or even $\{(a, b) \in \rmr: (a+bX)\cap G \text{ is finite} \}$)
can be taken to project to a $\lambda$-null set on the $b$-axis, namely,
$$\lambda\{b: \text{there exists  } a\in \re \text{  such that  } (a+bX) \subseteq (\re \setminus G)\} = 0,$$
as it has been shown by Kolountzakis \cite{Kol} in the case of bounded infinite sets~$X$.
\end{remark}

\begin{remark}
It should be pointed out that it is more important  to have such exceptional set that project to a null (or at least small) set in another directions (determined by~$X$). For example, assuming $0 \in \overline X$, if we manage to have an exceptional set that projects to a null set on the $a$-axis, this would imply that $X$ is non-universal. Universality of $X$ is also equivalent to the fact that 
there is a constant $C$ such that for all sets $G \subseteq \re,$
$$\lambda \{a \in [0, 1]: \{a\} \times [1, 2]  \subseteq L_X(G) \} \leq C\lambda(G).$$ 
Indeed, if the ratio $\lambda \{a \in [0, 1]: \{a\} \times [1, 2]  \subseteq L_X(G) \} /\lambda(G)$ can be made
arbitrarily large, then  for any given $\varepsilon>0$ one can construct (using the union of several random translation copies  of one set with that fraction  larger than $3\frac{|\ln (\varepsilon)|}{\varepsilon}$, and applying Lemma 8) a set $H$ with $\lambda(H)<\varepsilon$ and $\lambda \{a \in [0, 1]: \{a\} \times [1, 2]  \subseteq L_X(H) \}>1-\varepsilon.$ Using the fact that $0 \in \overline X$, we can add to $H$ that small set of $a$'s in $[0, 1]$ for which $\{a\} \times [1, 2]$ is not covered by $ L_X(H)$ and 
conclude non-universality of $X$ this way.

If $X$ is bounded and $0 \in \overline X$, universality of $X$ is equivalent to that
there is a constant C such that for all sets $G \subseteq \re,$
$$\lambda \{a \in \re: \{a\} \times [1, 2]  \subseteq L_X(G) \} \leq C   \lambda(G).$$ 
This  geometric measure theory characterization of universality (for $X$ bounded with $0 \in \overline X$) can be reformulated to the following harmonic analysis/function spaces version; namely that 
there is a constant C such that 
$$\int \inf_{1<b<2} \sup_{x \in X^*}  |f(a+bx)| \,da \leq C\int |f(a)|\, da$$
whenever $X^*$ is a finite subset of $X$ and $f$ is a continuous function on $\re$ with compact support.

This has been previously obtained by Bourgain \cite{B}. It should be mentioned that Lemma 1 in his paper suggests that the condition above is equivalent to universality for any bounded set $X \subseteq \re$, without assuming $0 \in \overline X$. It seems to be an omission; to prove that this is still true without assuming $0 \in \overline X$ would need to know that if we add any point (or finitely many) to an universal set, then this new set will be universal as well. But validity of such claim seems to be open at the moment.
\end{remark}

\section{The case of translation copies}

Several authors investigated the problem with similarity replaced by congruence. But if scaling is not allowed, then even for a finite set $X\subseteq \re$ of cardinality at least $2$ we no longer have a proof of existence of 'translation copies of ~$X$' in measurable set ~$C$ of positive measure near any density theorem; we can only hope for 'translations of ~$X$' in such a set if~$C$ is of nearly full measure in some intervals of diameter slightly larger than diameter of ~$X$.
Hence, it is plausible to address the problem of {\it universality (modulo translations) w.r.t.\ the sets with complement of sufficiently small measure}. One can also ask the question of whether for any zero-sequence $X\subseteq \re$ and any Lebesgue measurable set $C \subseteq \re$ of positive measure there exist
a translation of~$X$, $a+X$, such that~$C$ contains all but finitely many elements of $a+X$. (By the Lebesgue density theorem we get easily that for almost every point $a \in C$ the set $(a+X) \cap C$ is infinite.)
These questions turned out to have negative answers. Komj\'ath \cite{Kom} showed that for any given zero-sequence ~$X$ and $\epsilon>0$ there is a set $C \subset [0, 1]$ with $\lambda(C)>1-\epsilon$ possessing the property: if $a \in [0, 1]$, then  $(a+X) \setminus C$ is infinite.

Here we present the complete characterization of 'universal sets $X$' in this setting in full generality, hence without restricting to the bounded sets. It turns out that a set $X\subseteq \re$ is universal (modulo translations) in above mentioned meaning if and only if it is {\it  uniformly locally finite}. 

\medskip

{\bf Uniformly locally finite sets.} Let us call a set $X \subseteq \re$ {\it uniformly locally finite} if
$$\sup\{\card(X\cap [u, u+1]): u\in\re\} < \infty\,,$$
or, equivalently, for some finite constant $M$, $\card(X\cap [u, v]) \leq M|v-u|\,$ for any bounded interval $[u, v]$ with $|v-u|\geq 1$. Clearly, if $X$ is uniformly locally finite then $X \cup K$ is as well for any finite set $K$.
Each uniformly locally finite set is linearly bounded as well, but the opposite implication needn't be true. The results of this section can be easily generalized to more locally compact topological groups endowed with their Haar measure.

\begin{theorem}\label{strnast} Let $X\subseteq \re$. We can observe the following dichotomy concerning universality/non-universality modulo translations w.r.t. the sets with the complement of sufficiently small measure.
\begin{enumerate}
\item [(a)] If $X$ is uniformly locally finite then there is a constant $\varepsilon>0$ such that whenever $G\subseteq \re$ is a Lebesgue measurable set with  $\lambda(G)<\varepsilon$, then $\re \setminus G$ contains plenty of translations of $X$;  namely $\{a \in \re: (a+X) \subseteq (\re \setminus G)\}$ has infinite Lebesgue measure.

\item [(b)] If $X$ is not uniformly locally finite then for every $\varepsilon>0$ there is an open set $G\subseteq \re$ with $\lambda(G)<\varepsilon$ such that $(a+X) \cap G$ is infinite  for each $a \in \re$; equivalently,  its complement $ \re \setminus G$ doesn't contain any translation copy of $X \setminus K$, for any finite set $K$.
\end{enumerate}
\end{theorem}

\begin{proof}
(a) Let $X$ be uniformly locally finite, and $M$ be a constant such that $\card(X\cap [u, v]) \leq 
M|v-u|\,$ whenever $[u, v]$ is a bounded interval with $v-u$ sufficiently large. We will show that if $G\subseteq \re$ is a Lebesgue measurable set with  $\lambda(G)<\frac{1}{M}$, then the set $\{a \in \re: (a+X) \subseteq (\re \setminus G)\}$, which is exactly $\{a \in \re: a \notin G-X\}$,
contains a significant fraction of each  sufficiently large interval. We consider a set $G$ with $\lambda(G)<\frac{1}{M}$ and will estimate from above the measure $\lambda((G-X) \cap [u, v])$.
(We can assume that $G$ is open; otherwise we take its open superset of slightly larger measure.) Consider all components $G_i=(d_i, d_i+\varepsilon_i), i=1, 2, \dots, $ of the set $G.$ The set of those $x\in X$ for which the set $G_i-x=(d_i-x, d_i-x+\varepsilon_i)$ intersects  $[u, v]$ is exactly
 $X\cap (d_i-v, d_i+\varepsilon_i-u),$ and its cardinality is at most $M(v-u+\varepsilon_i)<M(v-u+\frac{1}{M})$. Consequently, $\lambda((G_i-X) \cap [u, v])<\lambda(G_i)M(v-u+\frac{1}{M})$, and after summing over all components, $\lambda((G-X) \cap [u, v])<\lambda(G)M(v-u+\frac{1}{M})$. Now, if we divide this inequality by $v-u$ and assume that $v-u$ is sufficiently large, then this fraction is less than $1-\tau$, as $\lambda(G)M<1-\tau$ for some $\tau>0$, by our assumptions about how small the measure of $G$ is. Hence $\lambda\{a \in [u, v]: (a+X) \subseteq (\re \setminus G)\}=\lambda\{a \in [u, v]: a \notin G-X\}>\tau|v-u|,$ that completes the proof of part (a).
 
 \medskip
 
(b) We could prove first an analogue of Lemma~10 under the weaker assumption of X not being uniformly locally finite that would apply specifically to translations; where the part (ii) of Lemma~10 is required to be fulfilled with $b=1$ only. This way we could prove an analogue of Proposition~\ref{eleven} for measures in $\re \times \{1\}$ and to complete the proof as that of Theorem~\ref{trinast}, showing that $\mu_X$ and $\tilde\mu_X$ vanish on $\re \times \{1\}$.

But the  problem to prove that  for $X$ that is not uniformly locally finite  $\mu_X$ vanishes on  $[0, 1] \times \{1\}$ (and, consequently, on $\re \times \{1\}$) can be further reduced as follows. One need to prove that for every $\varepsilon>0$  there is a set $H\subseteq \re$ with $\lambda(H)< 2\varepsilon$ such that $L_X(H)$ covers $[0, 1] \times \{1\}$ or, equivalently, that $H-X$ covers $[0, 1]$. But for that it suffices if for some $G$ with $\lambda(G)< \varepsilon$ the set  $G-X$ covers a significant fraction of $[0, 1]$, namely that $\lambda((G-X) \cap [0,1])>1-\varepsilon.$ Adding to $G$ a suitable translation copy of that yet uncovered set $[0,1] \setminus (G-X)$ will then produce the set $H$ with $\lambda(H)< 2\varepsilon$ and $H-X$ covering $[0, 1],$ as needed. So we reduced the problem to the following equivalent one:

{\bf Approximate covering problem.} Given $X\subseteq \re$ that is not uniformly locally finite and $\varepsilon>0$, show that there is a set $G\subseteq \re$ with $\lambda(G)<\varepsilon$ such that $\lambda((G-X) \cap [0,1])>1-\varepsilon$. 

Given such $X$ and $\varepsilon \in (0, 1)$, we choose an integer $n>3\frac{|\ln (\varepsilon)|}{\varepsilon}$.
Then we choose an interval $[u, u+1]$ such that  $\card(X\cap [u, u+1]) \geq n$, and then a set $Y \subset X\cap [u, u+1]$ of cardinality $n.$ Now we choose a number $\delta \in (0, \frac{\varepsilon}{8})$ for which $Y$ is $\delta$-separated, and put $I=(0, \delta)$. We apply Lemma~\ref{eight}, with $\nu$ being $1$-dimensional Hausdorff measure on $[0, 1] \times  \{1\}$, to several translations of a single set $L_Y(I)$.
(We will simplify the notation identifying $\re \times  \{1\}$ with $\re$ and using the Lebesgue measure.)
To follow closely notation used in Lemma~\ref{eight} let's denote an interval $[A, B]=[u-\delta, u+2]$ and  choose for  $k$ the largest integer for which $k \delta < \varepsilon$, and let $U=\prod^k_{i=1} [A, B].$ Observe that $[L_Y(I)]^1=I-Y$ and that, for every
$a \in [0, 1],$ $\lambda((I-Y)\cap[a-B, a-A])=n \delta,$ and $B-A=2+\delta$. So using Lemma~\ref{eight} we get that there exists $u=(u_1, u_2, \dots, u_k)\in U$ such that
$$\lambda \left([0, 1] \setminus \cup^k_{i=1} (u_i+(I-Y))\right)  \leq e^{-k\frac{n \delta}{2+\delta}}
< e^{-\frac{n \varepsilon}{3}} < \varepsilon,$$
where we used our choice of $\delta,$ $k$ and $n$ in last two inequalities. Hence the set $G=\cup^k_{i=1} (u_i+I)$ satisfies $\lambda(G) \leq k \delta < \varepsilon$ and $\lambda \left([0, 1] \setminus (G-X)\right)\leq \lambda \left([0, 1] \setminus (G-Y)\right) < \varepsilon,$ so $G$ solves our approximate covering problem, that completes the proof.
\end{proof}

\section{Non-universality of sets with large relatively $\delta$-separated subsets}\label{sec-five}

In this section we will give a probabilistic proof of non-universality of sets containing large relatively $\delta$-separated subsets. This will improve on the previous results of Kolountzakis \cite{Kol},
who proved non-universality of any bounded infinite set $X \subseteq \re$ which contains, for arbitrarily large $n$, a subset $\{x_1, x_2, \dots, x_n \}$ with $x_1>x_2> \dots  > x_n>0$ and
$$\min_{i=1,\dots,n-1} \frac{x_i-x_{i+1}}{x_1} \geq (1-o(1))^n.$$
In another words,
$-\log (\min_{i=1,\dots,n-1} \frac{x_i-x_{i+1}}{x_1}) = o(n)$, as $n \to +\infty$.
Despite the fact that a scale-invariant way has been used how to measure that minimum gap is large,
Kolountzakis'  sufficient condition of non-universality doesn't reflect enough that the problem is also translation invariant. We will improve on what we believe was suboptimal in the result above. The way how
the minimum gap $\min_{i=1,\dots,n-1} (x_i-x_{i+1})$ will be measured in our result is relative to $x_1-x_n$, rather than relative to $x_1$ in Kolountzakis' result. In fact, there are many zero-sequences in our Examples~\ref{five} that are arbitrarily relatively fine, and whose non-universality follows from our deterministic results given above, but where Kolountzakis' criterion does not apply. Our criterion rectifies this issue, and we show non-universality of any bounded infinite set $X \subseteq \re$ which contains, for arbitrarily large $n$, a subset $\{x_1, x_2, \dots, x_n \}$ with $x_1>x_2> \dots  > x_n$ and
$$\min_{i=1,\dots,n-1} \frac{x_i-x_{i+1}}{x_1-x_n} \geq (1-o(1))^n.$$


All previously published results on the Erd\"os similarity problem, except Bourgain's proof of non-universality of triple sum of infinite sets \cite{B}, follow easily from our theorem below.

Our sufficient condition that implies non-universality of a bounded infinite set  $X\subseteq \re$ can be expressed in terms of sizes of relatively $\delta$-separated subsets of $X$. If we denote, for $n>=3$,
$$\delta_n(X) = \sup \{\delta>0:  \text{ there is a relatively } \delta \text{-separated set } X^*  \subseteq X \text{ with cardinality }|X^*|=n \},$$
then that sufficient condition reads as follows: 
$\lim\inf_{n \to +\infty} \frac{-\log \delta_n(X)}{n} =0.$\\
And if we denote, for $\delta \in (0, 1),$ 
$$N_\delta(X)= \max \{ |X^*|:  X^*  \subseteq X \text{ is a relatively }
\delta \text{-separated set } \},$$
the condition reads as follows:\   
$\lim\sup_{\delta \to 0+} \frac{N_\delta(X)}{-\log \delta} =\infty.$
Consequently, any bounded universal set $X\subseteq \re$ has to satisfy $N_\delta(X) \leq C |\log \delta|$ for any sufficiently small $\delta>0$. From this one can easily derive that all universal sets have zero dimension (Hausdorff, or Minkowski). More precisely, their Hausdorff $H^h$ (or packing $P^h$) gauge measure is locally finite for the measure which is defined using gauge function $h(\delta)=\frac{1}{|\log \delta|}$, if $\delta>0$ is small.

\begin{theorem}\label{fifteen}
 Let  $X\subseteq \re$ be a bounded infinite set which contains, for arbitrarily small $\sigma>0$, a subset $\{x_1, x_2, \dots, x_n \}$ with $x_1>x_2> \dots  > x_n$,  $n \geq 3$, and
$$\min_{i=1,\dots,n-1} \frac{x_i-x_{i+1}}{x_1-x_n} \geq (1-\sigma)^n.$$
Then $X$ is non-universal (modulo similarities) w.r.t.\ the sets with complement of finite Lebesgue measure.  For every $\varepsilon >0$ there exists an open set $G\subseteq \re$ with
$\lambda (G)<\varepsilon $ such that any similar copy of $X$ intersects $G$ in an
infinite set. In another words, the set $ \re \setminus G$ doesn't contain any similar copy of $X \setminus K$, for any finite set~$K$. 
\end{theorem}

\begin{proof} We will show that the assumptions made about $X$ imply that $\mu_X\equiv0$ on $[0, 1] \times [1, 2]$ (and, consequently, $\mu_X\equiv0$ on $\re \times (\re\setminus\{0\})$), the same conclusion then follow with $X \setminus K$ in place of $X$, for any finite set $K$. That will complete the proof similarly as in the proof of Theorem~\ref{four}.\\
We will aim to prove that for every $\varepsilon>0$  there is a set $H\subseteq \re$ with $\lambda(H)< 2\varepsilon$ such that $L_X(H)$ covers $[0, 1] \times [1, 2]$. Keeping such an $\varepsilon>0$ fixed,
we choose a subset $X^*$ of $X$  as follows. First, fix a constant $C(X) =\sup\{|x|: x \in X \}$. For a positive constant  $q=\frac{\varepsilon}{3(1+4C(X))}$ we choose a subset $X^*=\{x_1, x_2, \dots, x_n \}$ of $X$ with $x_1>x_2> \dots  > x_n$,  $n \geq 3$, such that for $\delta$ defined by
$$\delta=\min_{i=1,\dots,n-1} \frac{x_i-x_{i+1}}{x_1-x_n},$$
(which by our assumptions can satisfy $-\log \delta = o(n)$ with $n \to \infty$) we have
$$\frac{- \log \delta}{n}+\frac{ \log n}{n}-\frac{\log \varepsilon}{n}+\frac{\log [12(1+C(X))]}{n}
\leq \frac{\varepsilon}{3(1+4C(X))}.$$ 
(Such choice is certainly possible, as the left hand side tends to $0$ with $n \to \infty$, and the right hand side is positive.)

We will see that $L_{X^*}(H)$  with a properly chosen set $H\subseteq \re$ for which $\lambda(H)< 2\varepsilon$ can cover $[0, 1] \times [1, 2]$. It is useful to study covering using a translation copy of this set, $Y=\{y_1, y_2, \dots, y_n \}=X^*-x_n$, with $y_i=x_i-x_n$ for $i=1,2,\dots,n$. Clearly, $y_1>y_2> \dots  > y_n=0.$\\
As it follows from Proposition 1(v), $L_{X^*}(H)$ will cover $[0, 1] \times [1, 2]$ iff $L_Y(H)$ will cover the parallelogram $\psi_{x_n}([0, 1] \times [1, 2]).$ This parallelogram is contained in the (smallest) rectangle $[A, B] \times [1, 2]$, which is either $[0, 1+x_n] \times [1, 2]$ if $x_n \geq 0$, or $[x_n, 1] \times [1, 2]$ if $x_n \leq 0$.

 We aim to cover the rectangle $[A, B] \times [1, 2]$ by $L_Y(H)$ using a set $H$ with $\lambda(H)< 2\varepsilon.$ We will find first a smaller set $G$ such that $L_Y(G)$  fully covers nearly all horizontal segments of $[A, B] \times [1, 2]$, and then we will construct $H$ by adding a small set to that $G$ in such a way that $L_Y(H)$ covers $[A, B] \times [1, 2]$.
 
Observe now that $L_Y(\{z\})$ intersects this rectangle iff $z \in [A, B+2y_1]$.

Denote $\tau= \min_{i=1,\dots,n-1} (x_i-x_{i+1})= \min_{i=1,\dots,n-1}( y_i-y_{i+1})$ and $\delta=\tau/(x_1-x_n)=\tau/y_1$.  We will use discretized version of that interval $[A, B+2y_1]$, namely smallest semiopen interval of the form $[m \tau, M\tau)$
with integers $m, M$ that contains $[A, B+2y_1]$. Observe that the length of the interval $[A, B]$ is at most $1+C(X)$, and the length of $[m \tau, M \tau)$ is at most $1+4C(X)$.\\
Consider intervals $I_j=[j\tau, (j+1)\tau), j=m,m+1,\dots,M-1$, and put each $I_j$ in $G$ independently of the other intervals with probability $q$ which will be determined later. In other words
$$ {\bf 1}_G=\sum _{j=m}^{M-1}  Z_j{\bf 1}_{[j\tau, (j+1)\tau)},$$
where $Z_j \in \{0, 1\}$ are independent indicator random variables with a fixed expected value ${\bf E}Z_j=q$. 
For a given $G$ we write for $a \in [A, B]$,
$$\phi_G(a)={\bf 1}_{[\exists\\  b \in [1, 2] \text{   such that  }  \\  (a+bY) \subseteq (\re \setminus G)]}
={\bf 1}_{[\exists\\  b \in [1, 2] \text{   such that  }  \\ (a, b) \notin L_Y(G)]}$$

We have ${\bf E} \lambda(G)=(M-m)\tau q$, and
$${\bf E}  \int_A^B \phi_G(a) da=\int_A^B {\bf Pr  }[\exists\\  b \in [1, 2] \text{   such that  }  \\  (a+bY) \subseteq (\re \setminus G)]    da.$$

In what follows we will prove that with properly chosen probability $q$ we can achieve that both expectations, ${\bf E} \lambda(G)$ and ${\bf E}  \int_A^B \phi_G(a) da$ are $\leq \varepsilon/3$.
From that one can then deduce, using standard large deviations arguments, that there exists a set $G$
for which $\lambda(G)<\varepsilon$ and  $\int_A^B \phi_G(a) da<\varepsilon.$\\
Now it is obvious how to construct a set $H$ with $\lambda(H)< 2\varepsilon$ for which $L_Y(H)$ covers $[A, B] \times [1, 2]$; we may add to a set $G$ that small exceptional set 
$\{a \in [A, B]: \exists  b \in [1, 2] \text{   such that  }   (a, b) \notin L_Y(G)\}.$\\

To complete the proof, we need to justify that both, ${\bf E} \lambda(G)$ and ${\bf E}  \int_A^B \phi_G(a) da$, can be made $\leq \varepsilon/3$ with proper choice of $q$. As ${\bf E} \lambda(G)=(M-m)\tau q \leq (1+4C(x))q$, we can take $q=\frac{\varepsilon}{3(1+4C(X))}$ to ensure that ${\bf E} \lambda(G) \leq \varepsilon/3$. Now we will estimate from above the value of ${\bf E}  \int_A^B \phi_G(a) da$.
\\

Fix $a \in [A, B].$ To check whether there exists $b \in [1, 2]$ such that $(a+bY) \subseteq (\re \setminus G)$ (equivalently, such that $(a, b)  \notin L_Y(G)$) it is sufficient to check whether such $b$ exists in a finite set $S(a)=\{b_1=1, b_2,\dots,b_N\}$ consisting of $1$ and those $b \in (1, 2]$ for which some $a+by_i, i=1,2,\dots,n-1$ is in the set $\{j\tau: m<j<M \}$ (as only those points $(a,b)$  can be boundary points of
the set $L_Y(G)$). For each $ i=1,2,\dots,n-1$ we have (approximately) $y_i/\tau$ such points
in $b \in (1, 2]$ for which $a+by_i$ corresponds to some $j\tau$ from the interval $(a+y_i, a+2y_i]$.
For any $a \in [A, B]$ such a  set $S(a)$ consists of at most $|S(a)| \leq 4ny_1/\tau=4n/\delta$ elements.\\

Since the length of the $I_j$ has been chosen so small we have that for each $(a, b) \in [A, B] \times [1, 2]$ the points $a+by_i, i=1,2,\dots,n,$ all belong to different intervals $I_j$. For this reason we have, for any fixed $(a, b) \in [A, B] \times [1, 2]$, 
$${\bf Pr  }[(a+bY) \subseteq (\re \setminus G)] ={\bf Pr  }[(a, b) \notin L_Y(G)]=(1-q)^n.$$
For any fixed $a \in [A, B]$, with the bound $|S(a)| \leq 4n/\delta$ on the number of points that it suffices to check we get
$${\bf Pr  }[\exists\\  b \in [1, 2] \text{   such that  }  \\  (a+bY) \subseteq (\re \setminus G)]  \leq |S(a)| (1-q)^n
\leq 4n(1-q)^n/\delta$$
and therefore $${\bf E}  \int_A^B \phi_G(a) da \leq 4(1+C(X))n(1-q)^n/\delta.$$ To check that  ${\bf E}  \int_A^B \phi_G(a) da \leq \varepsilon/3$ it is sufficient that $ 4(1+C(X))n(1-q)^n/\delta \leq \varepsilon/3$,
which is equivalent to (after applying $\log$ function, dividing by $n$, and rearranging),
$$\frac{- \log \delta}{n}+\frac{ \log n}{n}-\frac{\log \varepsilon}{n}+\frac{\log [12(1+C(X))]}{n}
\leq -\log(1-q),$$ 
which will follow from
$$\frac{- \log \delta}{n}+\frac{ \log n}{n}-\frac{\log \varepsilon}{n}+\frac{\log [12(1+C(X))]}{n} \leq q,$$
(using $q<- \log(1-q)$ valid for $q \in (0, 1)$); the last inequality is certainly true with $q=\frac{\varepsilon}{3(1+4C(X))}$ due to our choice of  $X^*$. That completes the proof.
\end{proof}

\bigskip

\begin{remark}  In our plane covering reformulation of the Erd\"os similarity problem  there is still a significant gap between what we know about the problem of "covering almost everything", and the problem of "covering everything". While  some classes of infinite sets are known to be provably non-universal, the Erd\"os similarity problem remains open. To the best of our knowledge, the universality/non-universality status is known for no zero-sequence $\{x_k\}$
that decreases geometrically or faster (i.e., $x_{k+1} \leq \rho x_k$ for some fixed $\rho <1$ and all $k$).\

In particular, it is (rather surprisingly) still unknown whether the geometric sequence $X=\{2^{-k}\}^\infty_{k=0}$ is universal or not. This seems a natural benchmark problem for any next progress on this question, as we can expect that the question of universality/non-universality in such simple explicit example has to be much easier to answer than in the general problem. For this explicit set $X$, one can think about many promising  ways of constructions coverings ( either deterministic or probabilistic)
by our bush-shape sets $L_X(G)$. Scaling invariance present in geometric sequences allows to prove (e.g., using Proposition 1(xi)) interesting properties of $\mu_X$ that are not available for more general $X$. For example, one can prove in this case the following:\

(i)   $\mu_{2^{-r}X}=\mu_X$ for any positive integer $r$,\

(ii)  $\mu_X (A \times B)=\lambda(A).\sigma(B)$ for an outer measure $\sigma$ in $\re\setminus\{0\}$ satisfying $\sigma (\alpha B)=\sigma (B)$ for each constant $\alpha \in \re\setminus\{0\}.$\

We hope that this paper will stimulate many new approaches to the problem.

\end{remark}

\end{document}